\documentclass[12pt,righttagstyle]{amsart}

\usepackage{amssymb,latexsym}

\def\C{{\mathbb C}}

\def\tub{T_\Omega}

\newtheorem{thm}{Theorem}[section]
\newtheorem{prop}[thm]{Proposition}
\newtheorem{cor}[thm]{Corollary}
\newtheorem{lem}[thm]{Lemma}
\newtheorem{defn}[thm]{Definition}
\newtheorem{remark}[thm]{Remark}

\begin{document}

\title[Duren-Carleson theorem]{The Duren-Carleson theorem in tube
  domains over symmetric cones}

\author[D. B\'ekoll\'e]{David B\'ekoll\'e}
\address{Department of Mathematics, Faculty of Science, University of Ngaound\'er\'e\\ P.O.Box 454, Ngaound\'er\'e, Cameroon }
\email{{\tt dbekolle@univ-ndere.cm}}
\author[B. F. Sehba]{Benoit F. Sehba}
\address{Department of Mathematics\\ University of Ghana\\ P.O. Box LG62, Legon, Accra, Ghana}
\email{{\tt bfsehba@ug.edu.gh}}
\author[E. L. Tchoundja]{Edgar L. Tchoundja}
\address{Department of Mathematics, Faculty of Science, University of Yaound\'e I\\ P.O.Box 812, Yaound\'e, Cameroon}
\email{{\tt tchoundjaedgar@yahoo.fr}}

\subjclass{} \keywords{Symmetric cones, Hardy spaces, Bergman
spaces.}

\begin{abstract}
In the setting of tube domains over symmetric cones,  we determine a necessary and sufficient condition on a Borel measure $\mu$ so that  the Hardy space $H^{p}, \hskip 2truemm 1\leq p < \infty,$ continuously embeds in  the weighted Lebesgue space $L^q (d\mu)$ with a larger exponent. Finally we use this result to characterize multipliers from $H^{2m}$ to Bergman spaces for every positive integer $m$.
\end{abstract}
\maketitle
\section{Introduction and statements of the results}

Let $T_\Omega=V+i\Omega$ be the tube domain over an irreducible
symmetric cone $\Omega$ in the complexification $V^{\mathbb{C}}$ of a
Euclidean space $V$ of dimension $n$. Following the notation of
\cite{FaKo} we denote the rank of the cone $\Omega$ by $r$ and by
$\Delta$ the determinant function of $V$. Letting $V=\mathbb{R}^n$, we
have as example of symmetric cone on $\mathbb{R}^n$ the Lorentz cone
$\Lambda_n$ which is a rank 2 cone defined for $n\ge 3$ by
$$\Lambda_n=\{(y_1,\cdots,y_n)\in \mathbb{R}^n: y_1^2-\cdots-y_n^2>0,\,\,\,y_1>0\};$$
the determinant function in this case is given by the Lorentz form
$$\Delta(y)=y_1^2-\cdots-y_n^2.$$
\vskip .2cm
For $0<q<\infty$ and $\nu \in \mathbb{R}$, let
$L^q_\nu(T_\Omega)=L^q(T_\Omega,\Delta^{\nu-\frac{n}{r}}(y)dxdy)$
denote the space of measurable functions $f$ 
satisfying the condition
$$\|f\|_{q,\nu}=||f||_{L^q_\nu(T_\Omega)}:=\left(\int_{T_\Omega}|f(x+iy)|^q\Delta^{\nu-\frac{n}{r}}(y)dxdy\right)^{1/q}<\infty.$$ 
Its closed subspace consisting of holomorphic functions in $T_\Omega$ is the weighted Bergman space $A^q_\nu(T_\Omega)$. This space
is not trivial i.e $A^q_\nu(T_\Omega)\neq \{0\}$ only for
$\nu>\frac{n}{r}-1$ (see \cite{DD}, cf. also \cite{BBGNPR}). The Bergman projector $P_\nu$ is the orthogonal projector from the Hilbert-Lebesgue space $L^2_\nu(T_\Omega)$ to its closed subspace $A^2_\nu(T_\Omega).$ The usual (unweighted) Bergman space
$A^q(T_\Omega)$ corresponds to the case $\nu=\frac{n}{r}$.

 Without loss of generality, we may assume that $V=\mathbb R^n$ endowed with the standard inner product, and we shall apply this notation in the rest of the paper. By $H^p(T_\Omega), \hskip 2truemm 0<p<\infty$, we denote the holomorphic Hardy space on
the tube domain that is the space of holomorphic functions $f$
such that
$$\|f\|_{H^p}=\left(\sup_{t\in\Omega}\int_{\mathbb{R}^n}
|f(x+it)|^p dx\right)^{1/p}<\infty.$$
Let $0< p, \hskip 1truemm q<\infty.$ Our purpose is to characterize those positive Borel measures $\mu$ on $T_\Omega$ for which the Hardy space $H^p(T_\Omega)$ is continuously embedded into the Lebesgue space $L^q (T_\Omega, d\mu).$ We recall that given two Banach spaces of functions $X$ and $Y$ with respective norms $\|\cdot\|_X$ and $\|\cdot\|_Y$, it is said that $X$ continuously  embeds into $Y$ ($X\hookrightarrow Y$), if there exists a constant $C>0$ such that for any $f\in X$, $$\|f\|_Y\leq C\|f\|_X.$$
Taking $X=H^p(T_\Omega)$ and $Y=L^q (T_\Omega, d\mu)$ in the last inequality, then  on the one hand, for all $0< p, \hskip 1truemm q<\infty, $ an obvious example of such a measure $\mu$ is the delta measure $\delta = \delta_{i\mathbf e}$ at the point $i\mathbf e,$ where $\mathbf e$ is a fixed point of $\Omega.$ On the other hand, testing on the functions 
$$G(z) = G_w (z) := [\Delta^{-\nu - \frac nr} (\frac {z-\bar w}{2i})]^{\frac 1q},$$
with $w=u+iv \in T_\Omega,$ we obtain that a necessary condition for the embedding of the Hardy space $H^p(T_\Omega)$ into the Lebesgue space $L^q (T_\Omega, d\mu)$ is the existence of a positive constant $C_{p, q, \mu}$ such that
\begin{equation}\label{nec}
\int_{T_\Omega} |\Delta^{-\nu - \frac nr} (\frac {z-\bar w}{2i})|d\mu (z) \leq C_{p, q, \mu}\Delta^{-(\nu + \frac nr) + \frac {nq}{rp}} (v)
\end{equation}
whenever 
$$(\nu + \frac nr)\frac pq > \frac {2n}r - 1.$$
 The first result is an extension  of a result due to O. Blasco \cite{Bl} (cf. also \cite{BH}) valid on the unit disc, to the case of $T_\Omega$.

\begin{thm}\label{Blasco}
 Let $\mu$ be a Borel measure on $T_\Omega.$ If $p, q, \nu$ are real numbers satisfying the
conditions
\begin{enumerate}
\item[(i)] $0 < p < q, \hskip 2truemm \frac qp > 2 - \frac rn,$
\item[(ii)] $(\nu + \frac nr) \frac pq > (\frac {2n}r - 1),$
\end{enumerate}
then
\begin{enumerate}
\item  $H^p (T_\Omega)$ continuously embeds in $A^q_{\frac nr (\frac qp -1)} (T_\Omega)$\\
if and only if
\item
the condition
(\ref{nec})
implies that $H^p (T_\Omega)$ continuously embeds in $L^q (T_\Omega, d\mu).$
\end{enumerate}
\end{thm}

\begin{remark}
\begin{enumerate}
\item
For $n=r=1$ (the case of the upper half-plane, $\Omega = (0, \infty)),$ assertion $2.$ of the theorem was proved by P. Duren \cite{D1} (cf. also \cite{D}), using a modification of the argument given by L. Carleson \cite{C} in the case $p=q=2;$ assertion $2.$ was proved earlier by Hardy and Littlewood \cite{HL}.
\item
 We restrict to the condition $q>p$ and even $\frac qp > 2 - \frac rn.$ Otherwise, the standard Bergman space $A^q_{\frac nr (\frac qp -1)} (T_\Omega)$ is trivial, that is $A^q_{\frac nr (\frac qp -1)} (T_\Omega)=\{0\}.$ The assertion 1. of the theorem is false. Nevertheless, we observed above that there are Borel measures $\mu$ on $T_\Omega$ such that $H^p \hookrightarrow L^q (T_\Omega, d\mu)$ (see the open question in Section 7).
\end{enumerate}
\end{remark}

In section 3, we shall prove Theorem 1.1 in a more general form where $\nu$ is a vector of $\mathbb R^r.$\\
\indent
Our next result is  the following Hardy-Littlewood Theorem.
\begin{thm}\label{main1} Let $4\leq p<\infty$. Then 
$H^2(T_\Omega)\hookrightarrow A^{p}_{\frac{n}{r}(\frac{p}{2}-1)}(T_\Omega)$.
\end{thm}
In the case where $r=2$, it is possible to go below the power $p=4$. We have exactly the following.
\begin{thm}\label{main2} Let $r=2$ and $n=3, 4, 5, 6.$ Then
\begin{enumerate}
\item
$H^2 (T_{\Lambda_3}) \hookrightarrow A^p_{\frac {3p}4 - \frac 32} (T_{\Lambda_3})$ for all $\frac 83 < p < 4.$
\item
$H^2 (T_{\Lambda_4}) \hookrightarrow A^p_{p-2} (T_{\Lambda_4})$ for all $3 < p < 4.$
\item
$H^2 (T_{\Lambda_5}) \hookrightarrow A^p_{\frac {5p}4 - \frac 52} (T_{\Lambda_5})$ for all $\frac {16}5 < p < 4.$
\item
$H^2 (T_{\Lambda_6}) \hookrightarrow A^p_{\frac {3p}2 - 3} (T_{\Lambda_6})$ for all $\frac {10}3 < p < 4.$
\end{enumerate}
\end{thm}

\begin{remark}
For every positive integer $m\geq 2,$ it is easy to see that the continuous embedding $H^2(T_{\tub})\hookrightarrow A^{p}_{(\frac{p}{2}-1)\frac{n}{r}}(\tub)$ implies the continuous embedding $H^{2m} (T_{\tub})\hookrightarrow A^{mp}_{(\frac{p}{2}-1)\frac{n}{r}}(\tub)$.
\end{remark}

Following this remark, we deduce the following corollary from Theorem 1.1, Theorem 1.3 and Theorem 1.4.

\begin{cor}
 Let $\mu$ be a positive Borel measure on $\tub$ and $p=2m \quad (m=1,2,\cdots)$ be a positive even number. If $q$ is a positive number satisfying one of the two following conditions
\begin{enumerate}
\item[(i)]
$r\geq 2, \hskip 2truemm n \geq 3$ and $2p \leq q < \infty;$
\item[(ii)]
$r= 2, \hskip 2truemm n = 3, \hskip 1truemm 4, \hskip 1truemm 5, \hskip 1truemm 6$ and $2p(1-\frac 1n) \leq q < 2p.$
\end{enumerate} 
and if $\nu$ is a real number satisfying the
condition $(\nu + \frac nr) \frac pq > (\frac {2n}r - 1),$
then the condition
(\ref{nec})
implies that $H^p (T_\Omega)$ continuously embeds in $L^q (T_\Omega, d\mu).$

\end{cor}

Recall that given two Banach spaces of analytic functions $X$ and $Y$ with respective norms $\|\cdot\|_X$ and $\|\cdot\|_Y$, we say an analytic function $G$ is a multiplier from $X$ to $Y$, if there exists a constant $C>0$ such that for any $F\in X$, $$\|FG\|_{Y}\leq C\|F\|_{X}.$$ We denote by $\mathcal{M}(X,Y)$ the set of multipliers from $X$ to $Y$.

Let $\alpha\in \mathbb{R}$. We denote by $H_\alpha^\infty(\tub)$, the Banach space of analytic functions $F$ on $\tub$ such that $$\|F\|_{\alpha,\infty}:=\sup_{z\in \tub}\Delta(\Im z)^\alpha |F(z)|<\infty.$$
In particular, for $\alpha = 0,$ the space $H_0^\infty(\tub)$ is the space $H^\infty$ of bounded holomorphic functions on $\tub.$
The above results allow us to obtain the following characterization of pointwise multipliers 
from $H^2(\tub)$ to $A_\nu^p(\tub)$.
\begin{thm}\label{main3} Let $4\leq p<\infty$, $\nu>\frac{n}{r}-1$. Define $\gamma=\frac{1}{p}(\nu+\frac{n}{r})-\frac{n}{2r}$. Then for any integer $m\geq 1$, the following assertions hold.
\begin{itemize}
\item[(a)] If $\gamma>0$, then $\mathcal{M}(H^{2m}(\tub),A_\nu^{pm}(\tub))=H_{\frac{\gamma}{m}}^\infty(\tub)$.
\item[(b)] If $\gamma=0$, then $\mathcal{M}(H^{2m}(\tub),A_\nu^{mp}(\tub))=H^\infty(\tub)$.
\item[(c)] If $\gamma<0$, then $\mathcal{M}(H^2(\tub),A_\nu^p(\tub))=\{0\}$
\end{itemize}
\end{thm}
For $p<4$, we have under further restrictions the following.

\begin{thm}\label{main4} Let $2(2-\frac{r}{n})<p<4$, $\nu>\frac{n}{r}-1$. Assume that  $P_{(\frac{p}{2}-1)\frac{n}{r}}$ is bounded on $L_{(\frac{p}{2}-1)\frac{n}{r}}^p(\tub)$. Define $\gamma=\frac{1}{p}(\nu+\frac{n}{r})-\frac{n}{2r}$. Then for any integer $m\geq 1$, the following assertions hold.
\begin{itemize}
\item[(a)] If $\gamma>0$, then $\mathcal{M}(H^{2m}(\tub),A_\nu^{mp}(\tub))=H_{\frac{\gamma}{m}}^\infty(\tub)$.
\item[(b)] If $\gamma=0$, then $\mathcal{M}(H^{2m}(\tub),A_\nu^{mp}(\tub))=H^\infty(\tub)$.
\item[(c)] If $\gamma<0$, then $\mathcal{M}(H^{2m}(\tub),A_\nu^{mp}(\tub))=\{0\}$
\end{itemize}
\end{thm}

Finally we particularize the previous problems to the tube domain over the light cone $\Lambda_n.$ We take advantage of the geometry of this cone to prove the following restricted Hardy-Littlewood Theorem. The point here is that the exponent $p$ is no more restricted to the set of even positive integers and the exponents $p$ and $q$ are just related by the inequality $1\leq p < q < \infty$. We say that a subset $B$ of the Lorentz cone $\Lambda_n$ is a restricted region with vertex at the origin $O$ if the Euclidean distance of any point of $B$ from $O$ is less that a multiple of the Euclidean distance of that point from the boundary of $\Lambda_n.$ We denote $T_B$ the tube domain over $B.$

\begin{thm}\label{HLR}
Let $1\leq p<q<\infty.$ Then given each restricted region $B$  of the Lorentz cone $\Lambda_n$ with vertex $O,$ there exists a positive constant $C_{p, q} {\bf (B)}$ such that 
$$\int_{T_{B}} |F(z)|^q \Delta^{\frac n2 (\frac qp - 2)} (y)dxdy \leq C_{p, q} {\bf (B)} ||F||_{H^p}^q$$
for all $F\in H^p (T_{\Lambda_n}).$
\end{thm}

The plan of this paper is as follows. In section 2, we present some preliminary results. The Blasco Theorem 1.1 is proved in section 3. The Hardy-Littlewood Theorems 1.3 and 1.4 are established in section 4. Section 5 is devoted to the proof of Theorem 1.6. The proof of the restricted Hardy-Littlewood Theorem \ref{HLR} is given in section 6 while in section 7, we pose some open questions related this work.

\begin{remark}
In \cite{shamoyan}, R. Shamoyan and M. Arsenovic investigated the continuous embedding of some generalized Hardy spaces $H^2_\mu (\tub)$ defined in \cite{G} into weighted mixed norm Bergman spaces. The space $H^p_\mu (\tub), \hskip 2truemm 0<p<\infty,$ consists of the functions $f$ holomorphic on $\tub,$ satisfying
$$||f||_{H^p_\mu} := (\sup \limits_{t\in \Omega} \int_{\partial \Omega} \int_{\mathbb R^n} |f(x+i(y+t)|^p dxd\mu (y))^{\frac 1p} <\infty.$$
In their study, the measure $\mu$ has the form 
$$d\mu_{\textbf s} (y) = \chi_\Omega (y)\frac {\Delta_{\textbf s} (y)}{\Gamma_\Omega (\textbf s)}\frac {dy}{\Delta^{\frac nr} (y)},$$
where $\textbf s = (s_1,\cdots,s_r)\in \mathbb R^r$ belongs to the so-called Wallach set and is such that $s_j > 0, \hskip 2truemm j=1,\cdots,r.$ The generalized power function $\Delta_{\textbf s}$ is defined at the beginning of section 2. Contrary to Bergman spaces, the measures $\mu=\mu_{\textbf s}$ studied by these authors have their support on the boundary $\partial \Omega$ of $\Omega.$ The usual Hardy spaces $H^p (\tub)$ correspond to the case where $\mu=\delta_0$ (the delta measure) and are outside their scope of application. Their proofs rely heavily on a Paley-Wiener characterization of functions in $H^2_\mu (\tub)$ proved in \cite{G}. We use the same tool in the proof of Theorem 1.3 (Theorem 4.1): we include this proof for completeness. The reader will also point out that particularly in Theorem \ref{HLR}, the exponents $1\leq p<q<\infty$ we consider are more general ($p$ may be different from $2).$ 
\end{remark}

\section{Preliminaries and useful results}
Materials of this section are essentially from \cite{FaKo}. We
give some definitions and useful results.

Let $\Omega$ be an irreducible open cone of rank r inside a vector
space $V$ of dimension n, endowed with an inner product $(.|.)$
for which $\Omega$ is self-dual.
Let $G(\Omega)$ be the group of transformations of $\Omega$, and
$G$ its identity component. It is well-known that there exists a
subgroup $H$ of $G$ acting simply transitively on $\Omega$, that
is every $y\in \Omega$ can be written uniquely as $y=g\mathbf e$
for some $g\in H$ and a fixed $\mathbf e\in \Omega$.

We recall that $\Omega$ induces in $V$ a structure of Euclidean
Jordan algebra with identity $\mathbf e$ such that $$\overline
\Omega=\{x^2: x\in V\}.$$ We can identify (since $\Omega$ is
irreducible) the inner product $(.|.)$ with the one given by the
trace on $V$: $$(x|y)=tr(xy),\,\,\,x,y\in V.$$ Let
$\{c_1,\cdots,c_r\}$ be a fixed Jordan frame in $V$ and
$$V=\oplus_{1\le i\le j\le r }V_{i,j}$$ be its associated Pierce
decomposition of $V$. We denote by
$\Delta_1(x),\cdots,\Delta_r(x)$ the principal minors of $x\in V$
with respect to the fixed Jordan frame $\{c_1,\cdots,c_r\}$. More
precisely, $\Delta_k(x)$ is the determinant of the projection
$P_kx$ of $x$ in the Jordan subalgebra $V^{(k)}=\oplus_{1\le i\le
j\le k }V_{i,j}$. We have $\Delta=\Delta_r$ and $\Delta_k(x)>0$,
$k=1,\cdots,r$, when $x\in \Omega$. The generalized power function
on $\Omega$ is defined as
$$\Delta_{\bf s} (x)=\Delta_1^{s_1-s_2}(x)\Delta_2^{s_2-s_3}(x)\cdots\Delta_r^{s_r}(x),\; x\in \Omega;\; {\bf s}\in\C^r.$$
 Since the principal minors $\Delta_k, \hskip 1truemm k=1,\cdots,r$ are polynomials in $V,$ they can be extended in a natural way to the complexification $V^{\mathbb C}$ of $V$ as holomorphic polynomials we shall denote $\Delta_k \left (\frac {x+iy}i \right ).$ It is known that these extensions are zero free on $T_\Omega.$ So the generalized power functions $\Delta_{\bf s}$ can also be extended as holomorphic functions $\Delta_{\bf s} \left (\frac {x+iy}i \right )$ on $T_\Omega.$

Next, we recall the definition of the generalized gamma function on
$\Omega$:
$$\Gamma_{\Omega}({\bf s})=\int_{\Omega}e^{-(\mathbf
{\underline
e}|\xi)}\Delta_{\bf s}(\xi)\Delta^{-n/r}(\xi)d\xi \quad \quad ({\bf s}=(s_1,\cdots,s_r)\in
\C^r).$$ We set $d:= \frac {2\left (\frac nr -1 \right )}{r-1}.$ This integral converges if and only if $\Re
s_j>(j-1)\frac{d}{2}$,\\ for all
$j=1,\cdots,r.$ Being in this case it is equal to:
$$\Gamma_{\Omega}({\bf s})=(2\pi)^{\frac{n-r}{2}}\prod_{j=1}^{r}\Gamma \left (s_j-(j-1)\frac{d}{2}\right )$$
(see Chapter VII of \cite{FaKo}).  For $\mathbf s = (s,\cdots,s), \hskip 2truemm s\in \mathbb C,$ we simply write $\Gamma_\Omega (s)$ instead of $\Gamma_{\Omega}({\bf s}).$\\
We also record the following lemma.

\begin{lem}\label{1.2} Let $s\in \C$ with $\Re
s >\frac nr -1.$ Then for all $y\in \Omega$
we have
$$\int_{\Omega}e^{-(y|\xi)}\Delta^{s-\frac nr}(\xi)d\xi=\Gamma_{\Omega}(s)\Delta^{-s}(y).$$ 
\end{lem}

The beta function of the symmetric cone $\Omega$ is defined by the
following integral:
$$B_\Omega(p,q)=\int_{\Omega\cap (\mathbf e-\Omega)}\Delta^{p-\frac{n}{r}}(x)\Delta^{q-\frac{n}{r}}(\mathbf e-x)dx,$$
where $p$ and $q$ are in $\C$. When $\Re p>\frac{n}{r}-1$ and $\Re q>\frac{n}{r}-1$,
the above integral converges absolutely and
$$B_\Omega(p,q)=\frac{\Gamma_{\Omega}(p)\Gamma_{\Omega}(q)}{\Gamma_{\Omega}(p+q)}$$
(see Theorem VII.1.7 in \cite{FaKo}).

\begin{lem}\label{1.1'} Let $p, q\in \C$ with $\Re
p >\frac nr -1$ and $\Re
q >\frac nr -1$. Then, for all $y\in \Omega$
we have
$$\int_{\Omega \cap (u-\Omega)} \Delta^{p-\frac{n}{r}}(x)\Delta^{q-\frac{n}{r}}(u-x)dx=B_\Omega(p,q) \Delta^{p+q-\frac nr} (u).$$
\end{lem}

The following is \cite[Proposition 3.5]{BBGNPR}.
\begin{lem}\label{lem:pointwiseestimberg}
Let $1\leq p<\infty$ and $\nu>\frac{n}{r}-1$. Then there is a constant $C>0$ such that for any $f\in A_\nu^p(T_\Omega)$ the following
pointwise estimate holds:
\begin{equation}\label{eq:pointwiseestimberg}|f(z)|\le C\Delta^{-\frac{1}{p}(\nu+\frac{n}{r})}(\Im z)\|f\|_{p,\nu},\,\,\,\textrm{for all}\,\,\, z\in T_\Omega.
\end{equation}
\end{lem}
We refer to \cite{DD} for the following, whose proof relies on the previous lemma.
\begin{lem}\label{lem:bergbergembed}
Let $1\leq p,q<\infty$, $\alpha,\beta>\frac{n}{r}-1$. Then $A_\alpha^p(T_\Omega)\hookrightarrow A_\beta^q(T_\Omega)$ if and only if $\frac{1}{p}(\alpha+\frac{n}{r})=\frac{1}{q}(\beta+\frac{n}{r})$.
\end{lem}
From the above lemma, we deduce that to prove Theorem \ref{main1}, it is enough to do this for $p=4$.

We will make use of Paley-Wiener theory in the next section to prove Theorem
\ref{main1} and Theorem \ref{main2}. The following can be found in \cite{FaKo}.
\begin{thm}\label{thm:PWH} For every $F\in H^2(T_\Omega)$ there
exists $f\in L^2(\Omega)$ such that
$$F(z)=\frac{1}{(2\pi)^{\frac{n}{2}}}\int_{\Omega}e^{i(z|\xi)}f(\xi)d\xi,\,\,\,\, z\in T_\Omega.$$
Conversely, if $f\in L^2(\Omega)$ then the integral above
converges absolutely to a function $F\in 
H^2(T_\Omega)$. In this case, $||F||_{
H^2}=||f||_{L^2(\Omega)}$.\end{thm} 

In the sequel, we write $V=\mathbb R^n.$ For the proofs of the following two lemmas, cf. e.g. \cite{Nana}. 

\begin{lem}\label{integralV}
Let ${\bf s} = (s_1,...,s_r) \in \mathbb R^r$ and define
$$I_{\bf s} (y) := \int_{\mathbb R^n} \left \vert \Delta_{-\bf s} \left (\frac {x+iy}{i}\right )\right \vert dx \hskip 2truemm {\rm for} \hskip 2truemm y\in \Omega.$$
Then $I_{\bf s} (y)$ is finite if and only if $\Re s_j > (r-j)\frac d2 + \frac nr.$ In this case, $$I_{\bf s} (y)=C({\bf s})(\Delta_{-\bf s} \Delta^{\frac nr}) (y).$$
Furthermore, the function $F(z) = F_w (z) = \Delta_{-\bf s} (\frac {z-\bar w}{2i}) \quad (w=u+iv \hskip 2truemm {\rm fixed \hskip 1truemm in} \hskip 2truemm T_\Omega)$ is  in $H^p(T_\Omega)$ whenever $\Re s_j > \frac {(r-j)\frac d2 + \frac nr}p.$ In this case, we have $$||F||_{H^p} = C(s, p)(\Delta_{-\bf s} \Delta^{\frac n{rp}}) (v).$$
\end{lem}

The expressions of the constants $C({\bf s})$ and $C(s, p)$ are in terms of generalized gamma functions on the cone $\Omega.$

\begin{lem}\label{integralcone}
Let $v\in \tub$ and ${\bf s}=(s_1,...,s_r), {\bf t} = (t_1,...,t_r) \in \mathbb C^r.$ The integral
$$\int_{\Omega} \Delta_{-\bf s} (y+v)\Delta_{\bf t} (y)dy$$
converges if $\Re t_j > (j-1)\frac d2 - \frac nr$ et $\Re (s_j -t_j) > \frac nr + (r-j)\frac d2.$ In this case this integral is equal to $C_{{\bf s}, {\bf t}} (\Delta_{-{\bf s}+{\bf t}} \Delta^{\frac nr}) (v).$
\end{lem}

We  denote as in \cite{BBGNPR}
$$L^2_{-\nu}(\Omega)=L^2(\Omega;
\Delta^{-\nu}(\xi)d\xi).$$ The following Paley-Wiener
characterization of the space $A^2_\nu(T_\Omega)$ can be found 
in \cite{FaKo}.
\begin{thm}\label{thm:PWB} For every $F\in A^2_\nu(T_\Omega)$ there
exists $f\in L^2_{-\nu}(\Omega)$ such that
$$F(z)=\frac{1}{(2\pi)^{\frac{n}{2}}}\int_{\Omega}e^{i(z|\xi)}f(\xi)d\xi,\,\,\,\, z\in T_\Omega.$$
Conversely, if $f\in L^2_{-\nu}(\Omega)$ then the integral above
converges absolutely to a function $F\in A^2_\nu(T_\Omega)$. In this case, $||F||_{p,\nu}=||f||_{L^2_{-\nu}}$.\end{thm}

The (weighted) Bergman projection $P_\nu$ is given by
$$P_\nu f(z)=\int_{\tub}K_\nu(z, w) f(w)
dV_\nu(w),
$$
where
 $K_\nu(z, w)=
 c_\nu\,\Delta^{-(\nu+\frac{n}{r})}(\frac {z-\overline {w}}{2i})$
is the Bergman kernel, i.e the reproducing kernel of  $A^2_\nu(T_\Omega)$ (see
\cite{FaKo}). Here, we use the notation
$dV_\nu(w):=\Delta^{\nu-\frac{n}{r}}(v) du\,dv$, where
$w=u+iv$ is an element of $T_\Omega$. For $\nu = \frac nr,$ we simply write $dV (w)$ instead of $dV_{\frac nr} (w).$The positive Bergman operator $P_\nu^+$ is defined by replacing the kernel function by its modulus in the definition of $P_\nu.$

In the particular case of the tube domain over the Lorentz cone $\Lambda_n$ on $\mathbb R^n,$ the following theorem is a consequence of results of \cite{BBGR} and the recent $l^2$-decoupling theorem of \cite{BD}.

\begin{thm}\label{prop:Bergprojboundedness}
Let $\nu > \frac n2 -1.$ Then the Bergman projector $P_\nu$ of $T_{\Lambda_n}$ admits a bounded extension on $L^p_\nu (T_{\Lambda_n})$ if and only if 
$$p'_\nu <p<p_\nu := \frac {\nu +n-1}{\frac n2 -1 }-\frac {(1-\nu)_+}{\frac n2 -1 }.$$
\end{thm}
For the other cases we recall the following partial result.
\noindent
\begin{thm} \cite{BBGR}, \cite{BBGRS}. Let $\Omega$ be a symmetric cone of rank $>2.$   
Let $\nu > \frac nr -1.$ Then the Bergman projector $P_\nu$ of $T_{\Omega}$ admits a bounded extension on $L^p_\nu (\Omega)$ if 
$$q'_\nu <p<q_\nu := 2+\frac {\nu}{\frac nr-1 }.$$
\end{thm}

We will sometimes face situations where the weight of the projection differs from the weight associated to the space. We then need the following result (see \cite{sehba}).
\begin{prop}\label{prop:positiveBergprojboundedness}
Let $1\leq p<\infty$, $\nu\in \mathbb{R}$, and $\mu>\frac{n}{r}-1$. Then $P_\mu^+$ is bounded on $L_\nu^p(T_\Omega)$ if and only if $1<p< q_\nu-1$ and $\mu p-\nu>\left(\frac{n}{r}-1\right)\max\{1,p-1\}$.
\end{prop}
\begin{defn}
The generalized wave operator $\Box$ on
the cone $\Omega$ 
is the differential operator of degree $r$ defined by the equality
$$\Box_x [e^{i(x|\xi)}] = \Delta (\xi)e^{i(x|\xi)} \hskip 2truemm {\rm where} \hskip 2truemm \xi \in \mathbb R^n.$$
When applied to a holomorphic function on $\tub,$ we have $\Box= \Box_z = \Box_x$ where $z=x+iy.$
\end{defn}

We observe with \cite{ BBGNPR, BBGRS, BBPR} the following.
\begin{thm}\label{lem:Boxreverseineq}
Let $1<p<\infty$ and $\nu>\frac{n}{r}-1$.
\begin{enumerate}
\item 
There exists a positive constant $C$ such that for every $F\in A^p_\nu,$
$$\|\Box F\|_{p,\nu+p}\leq C\|F\|_{p,\nu}.$$
\item If moreover $p\geq 2,$ the following two assertions are equivalent.
\begin{itemize}
\item[(i)]
$P_\nu$ is bounded on $L_\nu^p(\tub);$.
\item[(ii)]
For some positive integer $m,$ the differential operator $$\Box^{(m)} := \Box \circ ... \circ \Box \quad (m \hskip 2truemm {\rm  times}) \hskip 2truemm: A_\nu^p \rightarrow A^p_{\nu+mp}$$ is a bounded isomorphism.
\end{itemize}
\end{enumerate}
\end{thm}

Let us finish this section by the following result on complex interpolation of Bergman spaces of this setting.
\begin{prop}\label{prop:complexinterp} Let $1\leq p_0<p_1<\infty$, $\nu_0,\nu_1>\frac{n}{r}-1$. Assume that for some $\mu>\frac{n}{r}-1$, the projection $P_\mu$ is bounded on both $L_{\nu_0}^{p_0}(\tub)$ and $L_{\nu_1}^{p_1}(\tub)$. Then for any $\theta\in (0,1)$, the complex interpolation space $[A_{\nu_0}^{p_0},A_{\nu_1}^{p_1}]_\theta$ coincides with $A_{\nu}^{p}$ with equivalent norms, where $\frac{1}{p}=\frac{1-\theta}{p_0}+\frac{\theta}{p_1}$ and $\frac{\nu}{p}=\frac{1-\theta}{p_0}\nu_0+\frac{\theta}{p_1}\nu_1$.
\end{prop}
\begin{proof}
Consult e.g. \cite{BGN}.
\end{proof}
\section{Proof of the Blasco Theorem.}
\subsection{Proof of Theorem 1.1}
By Lemma \ref{integralV}, the function
$$G (z) = G_w (z) := \left (\Delta_1^{-\frac {\nu_1 - \nu_2}q}...\Delta_{r-1}^{-\frac {\nu_{r-1} - \nu_r}q}\Delta_r^{-\frac {\nu_r + \frac nr}q}\right )\left (\frac {z-\bar w}{2i}\right )$$
with $w=u+iv \in T_\Omega,$ belongs to $H^p (\tub)$ if and only if $$(\nu_j +\frac nr) \frac pq > (r-j)\frac d2 + \frac nr \hskip 2truemm (j=1,...,r).$$  Moreover 
$$||G||_{H^p (\tub)} = C_{p, q} \left (\Delta_1^{-\frac {\nu_1 - \nu_2}q}...\Delta_{r-1}^{-\frac {\nu_{r-1} - \nu_r}q}\Delta_r^{-\frac {{\nu_r + \frac nr}}q + \frac n{rp}}\right ) (v).$$
So for these $\nu,$ a necessary condition for the continuous embedding $H^p (\tub) \hookrightarrow L^q (\tub, d\mu)$ to hold is the existence of a positive constant $C_{p, q, \mu}$ such that for every $w=u+iv \in \tub$,
\begin{equation}\label{nec2}
L(w) 
\leq C_{p, q, \mu} \left (\Delta_1^{-(\nu_1 - \nu_2)}...\Delta_{r-1}^{-(\nu_{r-1} - \nu_r)}\Delta_r^{-\nu_r - \frac nr + \frac {nq}{rp}}\right ) (v)
\end{equation}
where $$L(w):=\int_{\tub} \left \vert \left (\Delta_1^{-(\nu_1 - \nu_2)}...\Delta_{r-1}^{-(\nu_{r-1} - \nu_r)}\Delta_r^{-\nu_r - \frac nr }\right )\left (\frac {z-\bar w}{2i}\right )\right \vert d\mu (z).$$
We state Theorem 1.1 in the following more general form.

\begin{thm}
 Let $\mu$ be a Borel measure on $T_\Omega.$ If $p, q$ are real numbers and $\nu = (\nu_1,...,\nu_r)$ is a a vector of $\mathbb R^r$ satisfying the
conditions
\begin{enumerate}
\item[(i)] $0 < p < q, \hskip 2truemm \frac qp > 2 - \frac rn,$
\item[(ii)] $(\nu_j +\frac nr) \frac pq > (r-j)\frac d2 + \frac nr \hskip 2truemm (j=1,...,r),$
\end{enumerate}
then
\begin{enumerate}
\item  $H^p (T_\Omega)$ continuously embeds in $A^q_{\frac nr (\frac qp -1)} (T_\Omega)$\\
if and only if
\item
the condition
(\ref{nec2})
implies that $H^p (T_\Omega)$ continuously embeds in $L^q (T_\Omega, d\mu).$
\end{enumerate}

\end{thm}

\begin{proof}
We first show the implication $2. \Rightarrow 1.$ We must prove that the measure $d\mu (x+iy) = \Delta^{\frac nr (\frac qp - 2)}(y)dxdy$ satisfies the estimate $(\ref{nec2}),$ i.e.
\begin{eqnarray*}
L &:=&\int_{\tub} |(\Delta_1^{-(\nu_1 - \nu_2)}...\Delta_{r-1}^{-(\nu_{r-1} - \nu_r)\Delta_r^{-\nu_r - \frac nr )}})(\frac {z-\bar w}{2i})|\Delta^{\frac nr (\frac qp - 2)}dxdy\\ &\leq& C_{p, q, \mu} (\Delta_1^{-(\nu_1 - \nu_2)}...\Delta_{r-1}^{-(\nu_{r-1} - \nu_r)}\Delta_r^{-\nu_r - \frac nr + \frac {nq}{rp}}) (v).
\end{eqnarray*}

By Lemma \ref{integralV}, we have
\begin{eqnarray*} && \int_{\mathbb R^n} |(\Delta_1^{-(\nu_1 - \nu_2)}...\Delta_{r-1}^{-(\nu_{r-1} - \nu_r)}\Delta_r^{-\nu_r - \frac nr })(\frac {x+iy-\bar w}{2i})|dx\\ &=& C_\nu \Delta_1^{-(\nu_1 - \nu_2)}...\Delta_{r-1}^{-(\nu_{r-1} - \nu_r)}\Delta_r^{-\nu_r} (y+v)
\end{eqnarray*}
whenever $\nu_j  > (r-j)\frac d2   \hskip 2truemm (j=1,...,r).$  Moreover by Lemma \ref{integralcone}, we have
\begin{eqnarray*} &&\int_\Omega (\Delta_1^{-(\nu_1 - \nu_2)}...\Delta_{r-1}^{-(\nu_{r-1} - \nu_r)}\Delta_r^{-\nu_r} (y+v))\Delta^{\frac nr (\frac qp - 2)}dy\\ &=& (\Delta_1^{-(\nu_1 - \nu_2)}...\Delta_{r-1}^{-(\nu_{r-1} - \nu_r)}\Delta_r^{-\nu_r + \frac nr (\frac qp - 1)})(v) 
\end{eqnarray*}
whenever $\frac qp > 2-\frac rn$ and $\nu_j > \frac nr (\frac qp - 1) + (r-j)\frac d2 \quad (j=1,...,r).$\\
We next show the implication $1. \Rightarrow 2.$ We shall use the following lemma.

\begin{lem}
Let $q>0$ and let $\nu = (\nu_1,...,\nu_r) \in \mathbb R^r.$ There exists a positive constant $C_{q, \nu}$ such that for every $F\in \mathcal Hol (\tub)$ we have
$$|F(z)|^q \leq C_{q, \nu}\int_{\tub} \frac {|F(u+iv)|^q \left (\Delta_1^{\nu_1 - \nu_2}...\Delta_{r-1}^{\nu_{r-1} - \nu_r}\Delta_r^{\nu_r -\frac nr}\right ) (v)}{\left \vert \left (\Delta_1^{\nu_1 - \nu_2}...\Delta_{r-1}^{\nu_{r-1} - \nu_r}\Delta_r^{\nu_r+\frac nr}\right ) \left (\frac {x+iy-\bar w}{2i}\right )\right \vert}dudv.$$ 
\end{lem}

\begin{proof}[Proof of the lemma]
We denote $B(\zeta, \rho)$ the Bergman ball with centre $\zeta$ and radius $\rho.$ Since $|F|^q$ is plurisubharmonic, we have
$$|F(ie)|^q \leq C\int_{B(ie, 1)} |F(u+iv)|^q \frac {dudv}{\Delta^{\frac {2n}r} (v)}.$$
Recall that  $\frac {dudv}{\Delta^{\frac {2n}r}  (v)}$ is the invariant measure on $\tub.$ Let $z\in \tub$ and let $g$  be an affine automorphism of $\tub$ such that $g(ie) = z.$ We have
\begin{eqnarray*} |F(z)|^q &=& |(F\circ g)(ie)|^q\\ &\leq& C\int_{B(ie, 1)} |(F\circ g)(u+iv)|^q \frac {dudv}{\Delta^{\frac {2n}r} (v)}\\ &=& C\int_{B(z, 1)} |F(u+iv)|^q \frac {dudv}{\Delta^{\frac {2n}r} (v)}.
\end{eqnarray*}
We recall that $|\Delta_j (\frac {z-\bar w}{2i})| \simeq \Delta_j (v)$ for all $w=u+iv \in B(z, 1).$ This implies that 
$$
\begin{array}{clcr}
|F(z)|^q &\leq C_{q, \nu}\int_{B(z, 1)} \frac {\left \vert F(u+iv)\right \vert^q \left (\Delta_1^{\nu_1 - \nu_2}...\Delta_{r-1}^{\nu_{r-1} - \nu_r}\Delta_r^{\nu_r -\frac nr}\right )(v)}{\left \vert \left (\Delta_1^{\nu_1 - \nu_2}...\Delta_{r-1}^{\nu_{r-1} - \nu_r}\Delta_r^{\nu_r+\frac nr}\right ) \left (\frac {x+iy-\bar w}{2i}\right )\right \vert}dudv\\ 
&\leq C_{q, \nu}\int_{\tub} \frac {\left \vert F(u+iv)\right \vert^q \left (\Delta_1^{\nu_1 - \nu_2}...\Delta_{r-1}^{\nu_{r-1} - \nu_r}\Delta_r^{\nu_r -\frac nr}\right )(v)}{\left \vert \left (\Delta_1^{\nu_1 - \nu_2}...\Delta_{r-1}^{\nu_{r-1} - \nu_r}\Delta_r^{\nu_r+\frac nr}\right ) \left (\frac {x+iy-\bar w}{2i}\right )\right \vert}dudv. 
\end{array}
$$
\end{proof}
Let us set 
$$I(w):=\int_{\tub} \frac {d\mu (z)}{|(\Delta_1^{\nu_1 - \nu_2}...\Delta_{r-1}^{\nu_{r-1} - \nu_r}\Delta_r^{\nu_r+\frac nr}) (\frac {x+iy-\bar w}{2i})|}$$
and recall that for $\nu=(\nu_1,\cdots,\nu_r)\in \mathbb{R}^r$, $$\Delta_{\nu-\frac{n}{r}}(v)=(\Delta_1^{\nu_1 - \nu_2}...\Delta_{r-1}^{\nu_{r-1} - \nu_r}\Delta_r^{\nu_r -\frac nr})(v).$$
Using the Fubini-Tonelli Theorem, it follows from the previous lemma and the condition (\ref{nec2}) that
\begin{eqnarray*}
\int_{\tub} |F (z)|^q d\mu (z) 
&\leq& C_{q, \nu}
\int_{\tub}I(u+iv)|F_ (u+iv)|^q \Delta_{\nu-\frac{n}{r}}(v)dudv\\
&\leq& C_{p, q, \mu}\int_{\tub} \Delta^{\frac nr (\frac qp -2)} (v)|F(u+iv)|^q dudv.
\end{eqnarray*}
An application of the assertion $1.$ of the theorem implies that
$$\int_{\tub} |F (z)|^q d\mu (z) \leq C_{p, q, \nu}||F||^q_{H^p}.$$
This finishes the proof of the implication $(ii) \Rightarrow (i).$

\end{proof}

\section{Proofs of the Hardy-Littlewood Theorems.}
\subsection{Proof of Theorem \ref{main1}.}
In view of Lemma \ref{lem:bergbergembed}, it is sufficient to show the following result. 
\begin{thm} We have that
$$H^2(T_\Omega)\hookrightarrow
A^{4}(T_\Omega).$$\end{thm} 
\begin{proof} Given
$F$ in $H^2(\tub)$, we would like to show that $F^2$
belongs to $A^{2}(\tub)$. By
Theorem \ref{thm:PWH} there exists $f\in L^2(\Omega)$ such that
$$F(z)=\int_{\Omega}e^{i(z|\xi)}f(\xi)d\xi \quad \quad
(z\in T_\Omega).$$ It follows that \begin{eqnarray*} F^2 (z) &=&
\int_{\Omega\times
\Omega}e^{i(z|\xi+t)}f(\xi)f(t)d\xi dt \\
 &=& \int_{\Omega}\int_{\Omega\cap
(u-\Omega)}e^{i(z|u)}f(u-\xi)f(\xi)d\xi du \\
 &=& \int_{\Omega}e^{i(z|u)}g(u)du,\end{eqnarray*}
where $$g(u)=\int_{\Omega\cap
(u-\Omega)}f(u-\xi)f(\xi)d\xi.$$
It follows from Theorem \ref{thm:PWB} that to conclude, we only have to
show that  $g\in L_{-\frac{n}{r}}^2(\Omega)$.

We first estimate $|g(u)|^2$. Using H\"older's inequality and Lemma \ref{1.1'}, we obtain
\begin{eqnarray*} |g(u)|^2 &\le& \left(\int_{\Omega\cap
(u-\Omega)}|f(u-\xi)||f(\xi)|d\xi\right)^2\\
&\le& \left(\int_{\Omega\cap
(u-\Omega)}|f(u-\xi)|^2|f(\xi)|^2d\xi\right)\times
\left(\int_{\Omega\cap
(u-\Omega)}d\xi\right)\\ &=&
C\Delta^{\frac{n}{r}}(u)\left(\int_{\Omega\cap
(u-\Omega)}|f(u-\xi)|^2|f(\xi)|^2d\xi\right).\end{eqnarray*}
More precisely we have $C=B\left (\frac nr, \frac nr \right )$ ($B$ is the beta function). It follows easily that
\begin{eqnarray*}\int_{\Omega}\Delta^{-\frac{n}{r}}(u)|g(u)|^2du &\le& C\int_{\Omega}\int_{\Omega\cap
(u-\Omega)}|f(u-\xi)|^2|f(\xi)|^2d\xi
du\\ &=& C||f||_{L^2(\Omega)}^4=C||F||_{H^2}^4.\end{eqnarray*} The proof is complete. 
\end{proof}

\subsection{Proof of Theorem \ref{main2}.}
With an application of assertion $1.$ of Theorem \ref{lem:Boxreverseineq}, we deduce the following that will be useful in the proof of Theorem \ref{main2}.
\begin{cor}\label{cor:BoxH2A4}
There exists a constant $C>0$ such that for any $F\in H^2(T_\Omega)$,
\begin{equation}\label{eq:BoxH2A4}
\left(\int_{\tub}|\Delta(\Im z)(\Box F)(z)|^4dV(z)\right)^{1/4}\leq C\|F\|_{H^2(T_\Omega)}.
\end{equation}
\end{cor}
The following is also needed in our proof of Theorem \ref{main2}.
\begin{prop}\label{prop:BoxH2A2}
For every positive integer $m$ such that $2m > \frac nr - 1,$  there exists a constant $C_m>0$ such that for any $F\in H^2(\tub)$,
\begin{equation}\label{eq:BoxH2A2}
\left(\int_{T_{\Lambda_3}}|(\Box^{(m)} F)(z)|^2\Delta(\Im z)^{2m-\frac{n}{r}}dV(z)\right)^{1/2}= C\|F\|_{H^2(\tub)}.
\end{equation} 
\end{prop}
\begin{proof}
Let $F\in H^2(\tub)$. Recall with Theorem \ref{thm:PWH} that there exists $f\in L^2(\Omega)$ such that
$$F(z)=\int_{\Omega}e^{i(z|\xi)}f(\xi)d\xi \quad \quad
(z\in T_{\Omega})$$
with $\|F\|_{H^2(\tub)}=\|f\|_{L^2(\Omega)}$. It follows that
$$\Box^{(m)} F(z)=\int_{\Omega}e^{i(z|\xi)}\Delta^m (\xi)f(\xi)d\xi.$$
Using the Plancherel's formula, we obtain
$$\int_{\mathbb{R}^n}|\Box^{(m)} F(x+iy)|^2dx=\int_{\Omega}e^{-2(y|\xi)}|f(\xi)|^2\Delta(\xi)^{2m}d\xi.$$
Integrating the latter with respect to $\Delta(y)^{2m-\frac{n}{r}}dy$ and using the definition of the gamma function, we obtain
\begin{eqnarray*}
I &:=& \int_{\tub}|\Delta^m (\Im z)(\Box^{(m)} F)(z)|^2\Delta(\Im z)^{-n/r}dV(z)\\ &=& \int_{\Omega}\int_{\mathbb{R}^n}|\Box^{(m)} F(x+iy)|^2\Delta(y)^{2m-\frac{n}{r}}dydx\\
&=& \int_{\Omega}|f(\xi)|^2\Delta(\xi)^{2m}\left(\int_{\Omega}e^{-2(y|\xi)}\Delta(y)^{2m-\frac{n}{r}}dy\right)d\xi\\
 &=& C_m\int_{\Omega}|f(\xi)|^2d\xi.
\end{eqnarray*}
The latter equality relies on the condition $s=2m > \frac nr - 1$ required in Lemma 2.1.
\end{proof}
We use Corollary \ref{cor:BoxH2A4} and Proposition \ref{prop:BoxH2A2} to deduce the following.
\begin{cor}\label{cor:BoxH2Apinterp} Let $p\in (2, 4).$ For every positive integer $m$ such that $2m > \frac nr - 1,$  there exists a constant $C_m>0$ such that for any $F\in H^2(\tub)$, 
\begin{equation}\label{eq:BoxH2Apinterp}
\left(\int_{\tub}|(\Box^{(m)} F)(z)|^p\Delta(\Im z)^{mp+(\frac{p}{2}-2)\frac{n}{r}}dV(z)\right)^{1/p}\leq C\|F\|_{H^2(\tub)}.
\end{equation}
\end{cor}
\begin{proof}
Note that by Corollary \ref{cor:BoxH2A4} and Proposition \ref{prop:BoxH2A2}, \hskip 1truemm $\Box^{(m)}$ defines a bounded operator from 
$H^2(\tub)$ to $A_{2m}^2(\tub)$ and from $H^2(\tub)$ to $A_{4m+\frac{n}{r}}^4(\tub)$  respectively. It follows by interpolation that $\Box^{(m)}$ is bounded from $H^2(\tub)$ to $[A_{2m}^2,A_{4m+\frac{n}{r}}^4]_\theta$, $\theta\in (0,1)$.
\vskip .1cm
It is easy to check (using Proposition \ref{prop:positiveBergprojboundedness}) that for $\mu > \frac nr - 1$ large, the projector $P_\mu$ is bounded on both $L_{2m}^2(\tub)$ and $L_{4m+\frac{n}{r}}^4(\tub)$. Thus by Proposition \ref{prop:complexinterp}, $[A_{2m}^2,A_{4m+\frac{n}{r}}^4]_\theta=A_{mp+(\frac{p}{2}-1)\frac{n}{r}}^p(\tub)$. The proof is complete.
\end{proof}

\begin{remark}
Referring to \cite{BBGRS}, we have shown that $H^2(\tub)$  continuously embeds into the holomorphic Besov space $\mathbb B^p_{\frac nr (\frac p2 -1)} (\tub)$ for all $2<p<4.$
\end{remark}
The following follows from Corollary \ref{cor:BoxH2Apinterp} and assertion $2.$ of Theorem \ref{lem:Boxreverseineq}.
\begin{thm}\label{thm:H2mApembed} 
Let $4-\frac{2r}{n}<p<4$. Assume that $P_{\frac nr (\frac p2 -1)}$ is bounded on $L_{\frac nr (\frac p2 -1)}^p(\tub)$. Then for any integer $m\geq 1$, $H^{2m}(\tub)\hookrightarrow A_{\frac nr (\frac p2 -1)}^{mp}(\tub)$

\end{thm} 
We can now prove Theorem \ref{main2}.
\begin{proof}[Proof of Theorem \ref{main2}] 
The condition $p>4-\frac {2r}n$ is necessary for the non triviality of $A^p_{\frac nr (\frac p2 -1)} (\tub).$ By Theorem \ref{thm:H2mApembed}, it is enough to check that the Bergman projector $P_{\frac nr (\frac p2 -1)}$ is bounded on $L^p_{\frac nr (\frac p2 -1)} (T_\Omega).$ In view of Theorem \ref{prop:Bergprojboundedness} we first suppose that $r=2$ and $\frac n2 (\frac p2 -1) < 1.$ The inequality $\frac n2 -1 < 1$ implies that $n=3.$ So we have the condition $\frac 12 < \frac 32 (\frac p2 -1) < 1,$ or equivalently $\frac 83 < p < \frac {10}3.$ By Theorem \ref{prop:Bergprojboundedness}, the Bergman projector $P_{\frac 32 (\frac p2 -1)}$ is bounded on $L^p_{\frac 32 (\frac p2 -1)} (T_{\Lambda_3})$ if $4-\frac {2r}n<p<3p-4,$ or equivalently if $p>4-\frac {2r}n.$ Thus the conclusion  of Theorem 1.2 is valid in the case $r=2, \hskip 2truemm n=3$ if $\frac 83 < p < \frac {10}3.$\\
Still for $r=2, \hskip 2truemm n=3,$ we next suppose that $\frac n2 (\frac p2 -1) \geq 1,$ or equivalently $p\geq \frac {10}3.$ By Theorem \ref{prop:Bergprojboundedness}, in this case, the Bergman projector $P_{\frac 32 (\frac p2 -1)}$ is bounded on $L^p_{\frac 32 (\frac p2 -1)} (T_{\Lambda_3)}$ if $2<p<\frac {3p}2 +1.$ This condition always holds. This finishes the proof of Theorem \ref{main2} in the case $n=3.$\\
We next suppose that $r=2$ and $n\geq 4.$ Then $\frac n2 - 1 \geq 1.$ The condition $\frac n2 (\frac p2 -1) > \frac n2 -1$ is equivalent to $p>\frac {4(n-1)}n.$ Moreover by Theorem \ref{prop:Bergprojboundedness},  the Bergman projector $P_{\frac n2 (\frac p2 -1)}$ is bounded on $L^p_{\frac n2 (\frac p2 -1)} (T_{\Lambda_n})$ if $\frac {4(n-1)}n<p<\frac {\frac n2 (\frac p2 - 1) +n-1}{\frac n2 -1},$ or equivalently if $\frac {4(n-1)}n<p< \frac {2n-4}{n-4}.$ We must have the inequality $\frac {4(n-1)}n < \frac {2n-4}{n-4},$ which holds if and only if $n\leq 6.$ This finishes the proof of the continuous embedding $H^2(T_\Omega)\hookrightarrow A^{p}_{\frac{n}{r}(\frac{p}{2}-1)}(T_\Omega)$ for $n=4, 5, 6$ respectively for all $3<p<4, \hskip 2truemm \frac {16}5 < p < 4$ and $\frac {10}3 < p < 4.$
\end{proof}



\section{Multipliers from Hardy spaces to Bergman spaces.}
Let us now prove Theorem \ref{main3}.
\begin{proof}[Proof of Theorem \ref{main3}] We recall that $\gamma=\frac{1}{p}(\nu+\frac{n}{r})-\frac{n}{2r}$. Let us start by proving the first assertion.

$(a)$: First assume that $G\in H_{\frac{\gamma}{m}}^\infty(\tub)$. Then using Theorem \ref{main1}, we obtain that for any $F\in H^2(\tub)$,
\begin{eqnarray*}
\int_{\tub}|F(z)G(z)|^{mp}dV_\nu(z) &\leq& \|G\|_{\frac{\gamma}{m},\infty}^p\int_{\tub}|F(z)|^{mp}\Delta(\Im z)^{(\frac{p}{2}-1)\frac{n}{r}-\frac{n}{r}}dV(z)\\ &\leq& C\|G\|_{\frac{\gamma}{m},\infty}^p\|F\|_{H^{2m}}^{mp}.
\end{eqnarray*}
Conversely, if $G\in \mathcal{M}(H^{2m}(\tub),A_\nu^{mp}(\tub))$, then by Lemma \ref{lem:pointwiseestimberg}, we have a constant $C>0$ such that for any $F\in H^{2m}(\tub)$,
\begin{equation}\label{eq:pointwisemultcond}
|F(z)G(z)|\leq C\Delta(\Im z)^{-\frac{1}{mp}(\nu+\frac{n}{r})}\|F\|_{H^{2m}},\,\,\,\textrm{for all}\,\,\, z\in \tub.
\end{equation}
We test (\ref{eq:pointwisemultcond}) with the function $F(z)=F_w(z)=\Delta(\Im w)^{\frac{n}{2mr}}\Delta(\frac{z-\bar{w}}{i})^{-\frac{n}{mr}}$ ($w$ fixed) which is uniformly in $H^{2m}(\tub)$ by Lemma 2.7 and obtain that there exists $C>0$ such that for all $z\in \tub$,
\begin{equation}\label{eq:pointwisemultcond1}
|G(z)|\Delta(\Im w)^{\frac{n}{2mr}}|\Delta(\frac{z-\bar{w}}{i})^{-\frac{n}{mr}}|\leq C\Delta(\Im z)^{-\frac{1}{mp}(\nu+\frac{n}{r})}.
\end{equation}
Taking in particular $z=w$ in (\ref{eq:pointwisemultcond1}), we obtain that
$$\Delta(\Im w)^{\frac{1}{mp}(\nu+\frac{n}{r})-\frac{n}{2mr}}|G(w)|\leq C$$
and the constant $C$ does not depend on $w$. Thus $G\in H_{\frac{\gamma}{m}}^\infty(\tub)$.

$(b)$: The proof of the necessity part follows as above. For the sufficiency, one observes that in this case, $\nu=(\frac{p}{2}-1)\frac{n}{r}$. It follows using Theorem \ref{main1} that
\begin{eqnarray*}
\int_{\tub}|F(z)G(z)|^{mp}dV_\nu(z) &\leq& \|G\|_{H^\infty}^{mp}\int_{\tub}|F(z)|^{mp}\Delta(\Im z)^{(\frac{p}{2}-1)\frac{n}{r}-\frac{n}{r}}dV(z)\\ &\leq& C\|G\|_{H^\infty}^{mp}\|F\|_{H^{2m}}^{mp}.
\end{eqnarray*}

$(c)$: It is clear that $0$ is multiplier from $H^{2m}(\tub)$ to $A_\nu^{mp}(\tub)$. Now assume that $\gamma<0$ and that $G\in \mathcal{M}(H^{2m}(\tub),A_\nu^{mp}(\tub))$. Then following exactly the same steps as in the proof of the necessity part in assertion $(a)$, we obtain that there is a constant $C>0$ such that for any $z\in \tub$,
$$|G(z)|\leq C\Delta(\Im z)^{\frac{n}{2mr}-\frac{1}{mp}(\nu+\frac{n}{r})}.$$ 
As $\frac{n}{2r}-\frac{1}{p}(\nu+\frac{n}{r})>0$, we obtain that the right hand side of the last inequality goes to $0$ as $\Delta(y)\rightarrow 0$. Hence $G(z)=0$ for all $z\in \tub.$
The proof is complete.
\end{proof}

\begin{proof}[Proof of Theorem \ref{main4}] This follows as above using Theorem \ref{thm:H2mApembed}.
\end{proof}

\section{The restricted Hardy-Littlewood Theorem}
In this section we prove Theorem \ref{HLR}.
We recall that the Lorentz cone $\Lambda_n$ is defined by $\Lambda_n := \{y=(y_1, y')\in \mathbb R^+ \times \mathbb R^{n-1}: y_1 > |y'|\}.$
We shall rely on the following geometrical lemma.

\begin{lem}\label{geo}
We write $d\mu (y) = \frac {\Delta^\beta (y)}{y_1^{2\beta} }dy.$ Then given $\beta > \frac n2 -1,$ there exists a positive constant $C=C_\beta$ such that
$$\mu (\{y\in \Lambda_n : \Delta^{\frac n2} (y) < \gamma\}) \leq C\gamma$$
for all $\gamma > 0.$
\end{lem}

\begin{proof}[Proof of the Lemma.] Using hyperbolic coordinates, an arbitrary point $y\in \Lambda_n$ can be written as
$$y=(r\hskip 1truemm cht, r\hskip 1truemm sht \hskip 1truemm \omega), \quad \quad r>0, \hskip 1truemm t\geq 0, \hskip 1truemm \omega \in \mathbb R^{n-1}, \hskip 1truemm |\omega|=1.$$
We use spherical coordinates to write $\omega$ as
$$\omega = (cos \hskip 1truemm \varphi, \hskip 1truemm sin \hskip 1truemm \varphi) \hskip 2truemm  {\rm with} \hskip 2truemm 0\leq \varphi \leq 2\pi \hskip 2truemm {\rm if} \hskip 2truemm n=3,$$
and
$$\omega = (cos \hskip 1truemm \varphi_1, \hskip 1truemm sin \hskip 1truemm \varphi_1 \hskip 1truemm cos \hskip 1truemm \varphi_2, ..., a(\varphi), b(\varphi))$$
where $$a(\varphi):=sin \hskip 1truemm \varphi_1 \hskip 1truemm sin \hskip 1truemm \varphi_2 ... sin \hskip 1truemm \varphi_{n-3}\hskip 1truemm cos \hskip 1truemm \varphi_{n-2},$$
$$b(\varphi):=sin \hskip 1truemm \varphi_1 \hskip 1truemm sin \hskip 1truemm \varphi_2 ... sin \hskip 1truemm \varphi_{n-3}\hskip 1truemm sin \hskip 1truemm \varphi_{n-2}$$
{\rm with} \hskip 2truemm  $0\leq \varphi_j \leq \pi \hskip 1truemm (j=1,...,n-3), \hskip 1truemm 0\leq \varphi_{n-2}\leq 2\pi \hskip 2truemm {\rm if} \hskip 2truemm n\geq 4.$\\
\vskip .1cm
We have $r^2 = \Delta (y)$ and the Jacobian $J_n$ of this change of coordinates has absolute value
$$
\begin{array}{clcr}
|J_n| &=&r^2 sh\hskip 1truemm t &{\rm if} \hskip 2truemm n=3\\
|J_n| &=&r^{n-1}sh^{n-2}\hskip 1truemm t \hskip 1truemm sin^{n-3} \hskip 1truemm \varphi_1...sin \hskip 1truemm \varphi_{n-3} &{\rm if} \hskip 2truemm n\geq 4.
\end{array}
$$
Now we obtain
\begin{eqnarray*}
&& \mu (\{y\in \Lambda_n: \Delta^{\frac n2} (y) < \gamma\})\\ &= &\int_{r^n < \gamma} r^{n-1}\frac {sh^{n-2}\hskip 1truemm t}{ch^{2\beta} \hskip 1truemm t } \hskip 1truemm sin^{n-3} \hskip 1truemm \varphi_1...sin \hskip 1truemm \varphi_{n-3} \hskip 1truemm drdtd\varphi_1...d\varphi_{n-3}d\varphi_{n-2}\\
&=& c_n \gamma \int_0^\infty \frac {sh^{n-2}\hskip 1truemm t}{ch^{2\beta} \hskip 1truemm t} dt.
\end{eqnarray*}
The latter integral converges when $\beta> \frac n2 -1.$ This finishes the proof of the lemma.
\end{proof}

Let $1< p<q<\infty$ and let $\beta > \frac n2 - 1.$ We denote by $A^q_{p, \beta} (T_{\Lambda_n})$ the weighted Bergman space on $T_{\Lambda_n}$ defined by
$$A^q_{p, \beta} (T_{\Lambda_n}) := Hol ((T_{\Lambda_n})\cap L^q (T_{\Lambda_n}, \frac {\Delta^{\frac n2 (\frac qp - 2)+\beta} (y)}{y_1^{2\beta}}dxdy).$$
Obviously this weighted Bergman space contains the standard weighted Bergman space $A^q_\nu (T_{\Lambda_n}), \hskip 2truemm \nu = \frac n2 (\frac qp -1).$\\
We deduce the following corollary.

\begin{cor}
The weighted Bergman space $A^q_{p, \beta} (T_{\Lambda_n})$ is not trivial i.e $$A^q_{p, \beta} (T_{\Lambda_n}) \neq \{0\}.$$
\end{cor} 

\begin{proof}[Proof of the Corollary.] We shall show that given $w=u+iv \in T_{\Lambda_n},$ the function $F(z) := \Delta^{-\frac \nu q} (\frac {z-\bar w}{2i})$ belongs to $A^q_{p, \beta} (T_{\Lambda_n})$ when $\nu$ is large. By Lemma 2.7 we obtain
$$\int_{\mathbb R^n} |F (x+iy)|^q dx = C(q, \nu)\Delta^{-\nu + \frac n2} (y+v)$$
if $\nu > n-1.$ In the notations of the previous lemma, we write again $d\mu (y) = \frac {\Delta^\beta (y)}{y_1^{2\beta} }dy.$ Furthermore
\begin{eqnarray*}
L &:=& \int_{T_{\Lambda_n}} |F (x+iy)|^q \frac {\Delta^{\frac n2 (\frac qp - 2)+\beta} (y)}{y_1^{2\beta}}dxdy)\\ &=& C(q, \nu)\int_{\Lambda_n}\Delta^{-\nu + \frac n2} (y+v)\Delta^{\frac n2 (\frac qp - 2)} (y)d\mu (y)\\
&=&C(q, \nu)\left\{\sum_{k=1}^\infty \int_{2^{-k} <\Delta^{\frac n2} (y) \leq 2^{-k+1}} + \sum_{k=0}^\infty \int_{2^{k} <\Delta^{\frac n2} (y) \leq 2^{k+1}}\right\}
\end{eqnarray*}
On the one hand we have
\begin{eqnarray*}
I &:=& \sum_{k=1}^\infty\int_{2^{-k} <\Delta^{\frac n2} (y) \leq   2^{-k+1}}\\ &\leq& \Delta^{-\nu +\frac n2} (v)\sum_{k=1}^\infty\int_{2^{-k} <\Delta^{\frac n2} (y) \leq 2^{-k+1}} \Delta^{\frac n2(\frac qp -2)} (y)d\mu (y)\\
&\leq&C\Delta^{-\nu +\frac n2} (v)\sum_{k=1}^\infty 2^{-k(\frac qp - 2)}\int_{\Delta^{\frac n2} (y) \leq 2^{-k+1}} d\mu (y)\\
&\leq&C_\beta \Delta^{-\nu +\frac n2} (v) \sum_{k=1}^\infty 2^{-k(\frac qp - 1)}.
\end{eqnarray*}
The latter  inequality follows by the previous lemma and the latter sum converges because $\frac qp > 1.$ On the other hand we have
\begin{eqnarray*}
\sum_{k=1}^\infty\int_{2^{k} <\Delta^{\frac n2} (y) \leq 2^{k+1}} &\leq&\sum_{k=1}^\infty\int_{2^{k} <\Delta^{\frac n2} (y) \leq 2^{k+1}} \Delta^{-\nu + \frac n2} (y)\Delta^{\frac n2(\frac qp -2)} (y)d\mu (y)\\
&=&\sum_{k=1}^\infty 2^{k(-\frac {2\nu}n +  \frac qp -1)}\int_{\Delta^{\frac n2} (y) \leq 2^{k+1}} d\mu (y)\\
&\leq&C_\beta  \sum_{k=1}^\infty 2^{k(-\frac {2\nu}n + \frac qp)}.
\end{eqnarray*}
The latter  inequality follows by the previous lemma and the latter sum converges if $\nu$ is chosen sufficiently large.
\end{proof}

We observe that for every $y\in \Lambda_n,$ we have $d(y, \partial \Lambda_n) = \Delta^{\frac 12} (y).$ For the proof of Theorem \ref{HLR} it suffices to show the following theorem.

\begin{thm}
Let $1< p<q<\infty.$ Then for each $\beta > \frac n2 - 1,$ there exists a positive constant $C_{p, q, \beta}$ such that 
$$\int_{T_{\Lambda_n}} |F(x+iy)|^q \frac {\Delta^{\frac n2 (\frac qp - 2)+\beta} (y)}{y_1^{2\beta}}dxdy \leq C_{p, q, \beta} ||F||_{H^p}^q$$
for all $F\in H^p (T_{\Lambda_n}).$
\end{thm}

\begin{proof}
In the sequel, the notation $||.||_p$ stands for the $L^p$-norm in $\mathbb R^n.$ We record the following well-known facts. For every $F\in H^p (T_{\Gamma_n}),\hskip 2truemm p\geq 1,$ the limit $f(x) = \lim \limits_{y\rightarrow 0, \hskip 2truemm y\in \Lambda_n} F(x+iy)$ exists in the $L^p$-norm; moreover if we call $P(f)$ the Poisson integral of $f$ defined by 
$$P(f) (x+iy)= : \int_{\mathbb R^n} \frac {\Delta^{\frac n2} (y)}{\left \vert \Delta^n \left (\frac {x+iy-\xi}i \right )\right \vert}f(\xi)d\xi,$$
we have $F=P(f)$ and $||F||_{H^p} = ||f||_p.$ So it is enough to prove that there exists a positive constant $C_{p, q, \beta}$ such that 
$$\int_{T_{\Lambda_n}} |P(f) (x+iy)|^q \frac {\Delta^{\frac n2 (\frac qp - 2)+\beta} (y)}{y_1^{2\beta}}dxdy \leq C_{p, q, \beta} ||f||_p^q.$$
We shall rely on the following lemma.

\begin{lem}\label{basic}
Given $1\leq s\leq q<\infty,$ there exists a positive constant $C_{q, s}$ such that
\begin{equation}\label{Y}
\left(\int_{\mathbb R^n} |P(f) (x+iy)|^q dx\right)^{\frac 1q} \leq C_{p, s} ||f||_s\Delta^{\frac n2(\frac 1q - \frac 1s)} (y)
\end{equation}
for all $y\in \Lambda_n$ and $f\in L^s(\mathbb R^n).$ 
\end{lem}

\begin{proof}[Proof of the lemma.]
We apply the Young convolution inequality with the parameter $t\geq 1$ defined by $\frac 1q = \frac 1s + \frac 1t -1.$ We obtain
\begin{eqnarray*}
\left(\int_{\mathbb R^n} |P(f) (x+iy)|^q dx\right)^{\frac 1q} &\leq& ||f||_s\Delta^{\frac n2}(y)\left(\int_{\mathbb R^n} \frac 1{\left \vert \Delta^n \left (\frac {x+iy}i \right )\right \vert^t}dx\right)^{\frac 1t}\\
&\leq& C_{q, s}||f||_s (\Delta^{\frac n2}\Delta^{-n+\frac n{2t}})(y)\\ &=& C_{q, s}||f||_s\Delta^{\frac n2 (\frac 1q - \frac 1s)}(y).
\end{eqnarray*}
The latter inequality follows by Lemma 2.7.
\end{proof}

We have
$$
\int_{T_{\Lambda_n}} |P(f) (x+iy)|^q \frac {\Delta^{\frac n2 (\frac qp - 2)+\beta} (y)}{y_1^{2\beta}}dxdy$$ 
$$=\int_{\Lambda_n} \frac {\Delta^{\frac n2 (\frac qp - 2)+\beta} (y)}{y_1^{2\beta}}(\int_{\mathbb R^n} |P(f) (x+iy)|^q dx)^{\frac pq} (\int_{\mathbb R^n} |P(f) (x+iy)|^q dx)^{1-\frac pq}dy$$
$$\leq C_{p, q}^p ||f||_p^{q-p} \int_{\Lambda_n} (\int_{\mathbb R^n} |P(f) (x+iy)|^q dx)^{\frac pq} \frac {\Delta^{-\frac {np}{2q}+\beta} (y)}{y_1^{2\beta}}dy
$$
where the latter inequality follows by estimate (\ref{Y}) of Lemma \ref{basic} with $s=p.$

We define the operator $S$ on $L^s (\mathbb R^n, dx), \hskip 2truemm 1\leq s \leq q,$ by
$$Sf(y) := \Delta^{-\frac n{2q}} (y)||P(f) (.+iy)||_q \quad \quad (y\in \Lambda_n).$$
We shall show that $S$ is a bounded operator from $L^p (\mathbb R^n ,dx)$ to\\ $L^p (\Lambda_n, \frac {\Delta^\beta (y)}{y_1^{2\beta}}dy).$ The conclusion will follow by the Marcinkiewicz interpolation Theorem if we can prove that $S$ is a weak-type $(1, 1)$ operator and a weak-type $(q, q)$ operator. The estimate (\ref{Y}) of Lemma \ref{basic} gives
\begin{eqnarray*}
\{y\in \Lambda_n: Sf(y) > \lambda\} &\subset& \{y\in \Lambda_n: C_{p, s} ||f||_s \Delta^{-\frac n{2s}} > \lambda\}\\
&=&\{y\in \Lambda_n: \Delta^{\frac n2} (y) < (\frac {C_{p, s}||f||_s}\lambda)^s\}.
\end{eqnarray*}
An application of Lemma \ref{geo}  concludes the proof of the theorem.
\end{proof}
\vskip 3truemm

\section{ Open questions} 
We pose here some questions that arise from this work and for which our methods do not give any answer.
\begin{itemize}
\item[(a)]Can Theorem \ref{main1} be extended to the interval $4-\frac {2r}n<p\leq 4$ in the following two cases?
\begin{itemize}
\item[1.] $r=2$ and $n\geq 7;$ 
\item[2.] $r\geq 3.$
\end{itemize}
\item[(b)]Can the restricted Hardy-Littlewood Theorem \ref{HLR} be extended to the entire Lorentz cone (unrestricted)? 
\item[(c)]Can these theorems be extended to general symmetric cones? 
\item[(d)]What happens when $1<\frac qp \leq 2-\frac rn$ in the Duren-Carleson Theorem? Is  assertion 2. of Theorem \ref{Blasco} valid in this case?
\end{itemize} 

\end{document}